\documentclass[10pt, reqno]{amsart}

\usepackage{amsmath,amssymb,amsthm,amsfonts,verbatim}
\usepackage{microtype}
\usepackage[all,2cell]{xy}
\usepackage{mathtools}
\usepackage{graphicx}
\usepackage{pinlabel}
\usepackage{hyperref}
\usepackage{mathrsfs}
\usepackage{color}
\usepackage{enumerate}
\usepackage{cite}
\usepackage{subcaption}
\usepackage{tikz}
\usetikzlibrary{cd}
\tikzcdset{every label/.append style = {font = \small}}

\CompileMatrices

\usepackage[top=1.2in,bottom=1.8in,left=1.2in,right=1.2in]{geometry}

\theoremstyle{plain}
\newtheorem{theorem}{Theorem}[section]
\newtheorem{maintheorem}{Theorem}

\newtheorem{proposition}[theorem]{Proposition}
\newtheorem{lemma}[theorem]{Lemma}

\theoremstyle{definition}
\newtheorem{definition}[theorem]{Definition}

\newtheorem{remark}[theorem]{Remark}

\newcommand{\nc}{\newcommand}
\nc{\dmo}{\DeclareMathOperator}

\nc{\Q}{\mathbb{Q}}
\nc{\F}{\mathbb{F}}
\nc{\R}{\mathbb{R}}
\nc{\Z}{\mathbb{Z}}
\nc{\C}{\mathbb{C}}
\nc{\Ell}{\mathcal{L}}
\nc{\M}{\mathcal{M}}
\nc{\K}{\mathcal{K}}
\nc{\I}{\mathcal{I}}
\nc{\T}{\mathcal T}
\nc{\U}{\mathcal U}
\nc{\disk}{\mathbb{D}}
\nc{\hyp}{\mathbb{H}}

\nc{\CP}{\mathbb{CP}}
\nc{\cS}{\mathcal{S}}
\dmo{\Mod}{Mod}
\dmo{\PMod}{PMod}
\dmo{\LMod}{LMod}
\dmo{\Diff}{Diff}
\dmo{\Homeo}{Homeo}
\dmo{\dist}{dist}
\dmo\BDiff{BDiff}
\dmo\SO{SO}
\dmo\Hom{Hom}
\dmo\SL{SL}
\dmo\Sp{Sp}
\dmo\rank{rank}
\dmo\sig{sig}
\dmo\Out{Out}
\dmo\Aut{Aut}
\dmo\Inn{Inn}
\dmo\GL{GL}
\dmo\PSL{PSL}
\dmo\BHomeo{BHomeo}
\dmo\EHomeo{EHomeo}
\dmo\EDiff{EDiff}
\nc\Sig{\Sigma}
\dmo\Teich{Teich}
\dmo\Fix{Fix}
\nc{\pair}[1]{\langle #1 \rangle}
\nc{\abs}[1]{\left| #1 \right|}
\nc{\action}{\circlearrowright}
\nc{\norm}[1]{\left | \left | #1 \right | \right |}
\nc{\abcd}[4]{\left(\begin{array}{cc} #1 & #2 \\ #3 & #4 \end{array}\right)}
\nc{\into}\hookrightarrow
\dmo{\Isom}{Isom}
\nc{\normal}{\vartriangleleft}
\dmo{\Vol}{Vol}
\dmo{\im}{Im}
\dmo{\Push}{Push}
\dmo{\Conf}{Conf}
\dmo{\PConf}{PConf}
\dmo{\id}{id}
\dmo{\Jac}{Jac}
\dmo{\Pic}{Pic}
\dmo{\Stab}{Stab}
\dmo{\Arf}{Arf}
\dmo{\End}{End}
\dmo{\Gal}{Gal}
\dmo{\lcm}{lcm}
\dmo{\ab}{ab}
\dmo{\opp}{op}
\dmo{\SU}{SU}
\dmo{\OT}{\Omega \mathcal{T}}
\dmo{\OM}{\Omega \mathcal{M}}
\dmo{\spin}{spin}
\dmo{\even}{even}
\dmo{\odd}{odd}
\dmo{\comp}{\mathcal{H}}
\dmo{\Mgk}{\mathcal{M}_{g, \underline{\kappa}}}
\dmo{\orb}{orb}
\dmo{\AJ}{AJ}
\dmo{\Ck}{\mathsf{C}(\underline{\kappa})}
\dmo{\ord}{ord}
\dmo{\Int}{Int}
\dmo{\pr}{pr}
\dmo{\lab}{lab}
\dmo{\Sym}{Sym}
\dmo{\Div}{Div}
\dmo{\RelAut}{PRelAut}
\dmo{\RelArf}{RelArf}
\dmo{\PB}{PB}
\dmo\PAut{PAut}

\nc{\Span}[1]{\operatorname{Span}(#1)}
\newcommand{\onto}{\twoheadrightarrow}

\newcommand{\sing}{\underline{\kappa}}
\renewcommand{\epsilon}{\varepsilon}
\renewcommand{\tilde}{\widetilde}
\renewcommand{\le}{\leqslant}
\nc{\coloneq}{\mathrel{\mathop:}\mkern-1.2mu=}
\nc{\margin}[1]{\marginpar{\scriptsize #1}}
\nc{\para}[1]{\medskip\noindent\textbf{#1.}}
\nc{\red}[1]{\textcolor{red}{#1}}
\nc{\blue}[1]{\textcolor{blue}{#1}}
\nc{\proofof}[1]{\noindent {\em Proof (of #1).}}

\nc{\lb}{[}
\nc{\rb}{]}

\title[Relative homological representations of framed mapping class groups]{Relative homological representations of \\framed mapping class groups}

\author{Aaron Calderon and Nick Salter}
\email{aaron.calderon@yale.edu}
\email{nks@math.columbia.edu}
\thanks{AC is supported by NSF Award No. DGE-1122492. NS is supported by NSF Award No. DMS-1703181.}

\address{AC: Department of Mathematics, Yale University, 10 Hillhouse Ave, New Haven, CT 06511}
\address{NS: Department of Mathematics, Columbia University, 2990 Broadway, New York, NY 10027}
\date{June 25, 2020}

\begin{document}
\begin{abstract}
Let $\Sigma$ be a surface with either boundary or marked points, equipped with an arbitrary framing. In this note we determine the action of the associated ``framed mapping class group'' on the homology of $\Sigma$ relative to its boundary (respectively marked points), describing the image as the kernel of a certain crossed homomorphism related to classical spin structures. Applying recent work of the authors, we use this to describe the monodromy action of the orbifold fundamental group of a stratum of abelian differentials on the relative periods.
\end{abstract}
\vspace*{-2.5em}
\maketitle
\vspace*{-1em}

\section{Introduction}
Let $(\Sigma_g,Z)$ be a surface endowed with a nonempty finite set of marked points; we assume throughout that $g \ge 2$ unless otherwise specified. A {\em framing} of $(\Sigma_g,Z)$ is a trivialization of the tangent bundle of $\Sigma_g\setminus Z$; up to homotopy this is specified by a vector field vanishing only at $Z$. We say that two framings $\phi$ and $\psi$ are {\em isotopic} if the corresponding vector fields are isotopic through vector fields vanishing only at $Z$. The (pure) mapping class group $\PMod(\Sigma_g,Z)$ of the marked surface $(\Sigma_g,Z)$ admits a well--defined action on the set of isotopy classes of framings, and we define the {\em framed mapping class group} as the stabilizer of a chosen (isotopy class of) framing $\phi$:
\[
\PMod(\Sigma_g,Z)[\phi] := \{f \in \PMod(\Sigma_g,Z) \mid f\cdot \phi = \phi \mbox{ up to isotopy}\}.
\]
One of the most basic tools in the study of mapping class groups is the {\em homological representation} via its action on the first homology of the surface. In the presence of marked points, we can define the {\em relative homological representation}
\[
\Psi^{rel}: \PMod(\Sigma_g, Z) \to \PAut(H_1(\Sigma_g, Z;\Z)),
\]
where $\PAut(H_1(\Sigma_g, Z;\Z))$ is the ``pure automorphism group'' of $H_1(\Sigma_g, Z;\Z)$; see Section \ref{section:structure}. In this note we determine the action of $\PMod(\Sigma_g,Z)[\phi]$ on $H_1(\Sigma_g,Z;\Z)$ via $\Psi^{rel}$.

Recall that a {\em crossed homomorphism} is a map $f: G \to A$ where $G$ is a group and $A$ is a $\Z[G]$--module such that $f(g_1g_2) = f(g_1) + g_1 f(g_2)$ for all elements $g_1, g_2 \in G$. Kernels of crossed homomorphisms are subgroups but are not necessarily {\em normal} subgroups. 

\begin{maintheorem}\label{theorem:main}
For $g \ge 2$, let $\phi$ be a framing of the marked surface $(\Sigma_g, Z)$. Then there is a crossed homomorphism 
\[
\Theta_\phi: \PAut(H_1(\Sigma_g, Z; \Z)) \to H^1(\Sigma_g;\Z/2\Z)
\]
such that 
\[
\Psi^{rel}(\PMod(\Sigma_g, Z)[\phi]) = \ker (\Theta_\phi).
\]
\end{maintheorem}

An explicit description of $\Theta_\phi$ appears at the end of the introduction; see Section \ref{section:Xhom} for full details.

\para{Strata of translation surfaces}
Our main application of Theorem \ref{theorem:main} is to give an explicit description of the homological monodromy groups of strata of abelian differentials. To formulate our results, we recall that the moduli space $\Omega \mathcal M_g$ of genus $g$ abelian differentials is divided into {\em strata} according to the combinatorics of the zero locus. If $\sing = (\kappa_1, \dots, \kappa_n)$ is a partition of $2g-2$ then we use $\Omega \mathcal M_g(\sing)$ to denote the set of all abelian differentials with $n$ zeros of orders $\kappa_1, \ldots, \kappa_n$.

Over every connected component $\mathcal H$ of a stratum one can define a vector bundle (in the orbifold sense) $H_1^{\textit{rel}}$ whose fiber over a manifold point $(X, \omega)$ is the relative homology group $H_1(X, \Div(\omega); \R)$. The (orbifold) fundamental group of $\mathcal H$ therefore admits a monodromy action on this bundle
\[
\rho_H : \pi_1^{\text{orb}} ( \mathcal H) \rightarrow \PAut(H_1(X, \Div(\omega); \R)).
\]
We observe that $\pi_1^{\text{orb}} ( \mathcal H)$ may permute the zeros of $\Div(\omega)$; set $\widehat{\mathcal H} \rightarrow \mathcal H$ to be the finite, connected
\footnote{See, e.g., \cite[Proposition 4.1]{Boissy_Rauzy} for a proof that this cover is connected.}
cover of $\mathcal H$ associated to the kernel of this action. The bundle $H_1^{\text{rel}}$ pulls back under this covering, and we let $\Gamma_{\widehat{\mathcal H}}$ denote the image of the monodromy homomorphism restricted to $\pi_1^{\text{orb}} (\widehat{\mathcal H})$.

By deep results of Eskin, Filip, and Wright \cite{EFW_HullKZcocycle}, the Zariski closure of $\Gamma_{\widehat{\mathcal H}}$ is equal to $\Sp(2g, \R) \ltimes \R^{n-1}$ (i.e., it is ``as big as possible'' given the constraint arising from the intersection pairing on absolute homology). The action of $\Gamma_{\widehat{\mathcal H}}$ on absolute homology was determined by Gutierrez-Romo \cite{GR_RVgroups}, but explicit characterizations of the full group $\Gamma_{\widehat{\mathcal H}}$ were only known for hyperelliptic components of strata \cite{AMY_hyperRV}[Corollary 2.8] and for the non-hyperelliptic components of $\Omega \mathcal M_g(g-1, g-1)$ (and $\Omega \mathcal M_g(2g-2)$) \cite{GR_RVgroups}[Theorem 5.1].

Together with recent work of the authors computing the image of an associated ``topological monodromy homomorphism'' (see just below), Theorem \ref{theorem:main} allows us to generalize the computations listed above, identifying $\Gamma_{\widehat{\mathcal H}}$ in terms of the crossed homomorphism $\Theta_\phi$.

\begin{maintheorem}\label{theorem:corollary}
Let $\sing = (\kappa_1, \dots, \kappa_n)$ be a partition of $2g-2$ with $g \ge 5$ and let $\mathcal H$ be a non-hyperelliptic component of the stratum $\Omega\mathcal M_g(\sing)$. 
Let $\widehat{\mathcal H}$ be the cover of $\mathcal H$ corresponding to the kernel of the permutation action on the zeros and choose a basepoint $(X, \omega) \in  \widehat{\mathcal H}$. Let $\phi$ be the induced framing of $(X, \Div(\omega))$.
Then the homological monodromy group $\Gamma_{\widehat{\mathcal H}} \le \PAut(H_1(X, \Div(\omega); \Z))$ is computed to be
\[
\Gamma_{\widehat{\mathcal H}} = \ker (\Theta_\phi)
\]
for $\Theta_\phi$ the crossed homomorphism of Theorem \ref{theorem:main}.
\end{maintheorem}
\begin{proof}
We observe that there is a family of smooth curves $\mathcal X \rightarrow \widehat{\mathcal H}$ whose fiber over $(X, \omega)$ is $X$; the monodromy of this family therefore gives rise to a {\em topological monodromy homomorphism}
\[\rho: \pi_1^{\text{orb}} ( \widehat{\mathcal H}) \rightarrow \PMod(X, \Div(\omega)).\]
Theorem A of \cite{strata3} computes that the image of $\rho$ is exactly $\PMod(X, \Div(\omega))[\phi]$ (for $g \ge 5$). Now we observe that the homological monodromy representation factors through the topological monodromy via $\rho_H = \Psi^{rel} \circ \rho$. Applying Theorem \ref{theorem:main} yields the desired statement.
\end{proof}

\para{Relatively framed mapping class groups} Theorem \ref{theorem:main} is deduced from a somewhat stronger statement, which can also be used to supply some more information about strata. Suppose now that $\Sigma_{g,n}$ is a surface with $n \ge 1$ boundary components; then a nonvanishing vector field on $\Sigma_{g,n}$ gives rise to a framing $\phi$ of $\Sigma_{g,n}$. We say that framings $\phi$ and $\psi$ are {\em relatively isotopic} if the associated vector fields are isotopic through an isotopy which is trivial on $\partial \Sigma_{g,n}$. The mapping class group $\Mod(\Sigma_{g,n})$ admits a well-defined action on the set of relative isotopy classes of framings, and we can define the {\em relatively framed mapping class group} as the stabilizer of a chosen relative framing:
\[
\Mod(\Sigma_{g,n})[\phi] := \{f \in \Mod(\Sigma_{g,n}) \mid f \cdot \phi = \phi \mbox{ up to {\em relative} isotopy}\}.
\]
In the case of a surface with boundary, we consider the relative homological representation as follows:
\[
\Psi^{rel}: \Mod(\Sigma_{g,n}) \to \PAut(H_1(\Sigma_{g,n}, \partial \Sigma_{g,n};\Z)).
\]
Note that there is a natural isomorphism $p_*: H_1(\Sigma_{g,n}, \partial \Sigma_{g,n};\Z) \cong H_1(\Sigma_g,Z;\Z)$ induced by contracting each boundary component to a marked point which extends to an isomorphism (also denoted $p_*$) of the corresponding (pure) automorphism groups. The notion of relative isotopy is genuinely more restrictive than standard isotopy, and the relatively framed mapping class group is ``smaller'' than its absolute counterpart (see Section \ref{section:framings} for details). Despite this, we find that there are no further restrictions on the action on relative homology. 

\begin{maintheorem}\label{theorem:main2}
For $g \ge 2$, let $\phi$ be a relative framing of $\Sigma_{g,n}$. Then
\[
\Psi^{rel}(\Mod(\Sigma_{g,n})[\phi]) = \ker (\Theta_\phi \circ p_*)
\]
where $\Theta_\phi$ is the crossed homomorphism of Theorem \ref{theorem:main}.
\end{maintheorem}

\begin{remark}
One may use Theorem \ref{theorem:main2} together with the analysis of \cite[\S7]{strata3} to deduce that $\Gamma_{\widehat{\mathcal H}}$ is generated by the action of {\em cylinder shears}, certain deformations of abelian differentials along embedded Euclidean cylinders.
\end{remark}

\begin{remark}
Using Theorem \ref{theorem:main2} together with Theorem 7.13 of \cite{strata3}, one can also identify the homological monodromy groups of either of the non-hyperelliptic components of strata of {\em prong--marked} abelian differentials (see \cite[\S7.3]{strata3}) with $\ker (\Theta_\phi)$. We leave it to the reader to formulate and prove this (completely analogous) statement.
\end{remark}

\para{The crossed homomorphism $\Theta_\phi$}
We now give an explicit description of the crossed homomorphism which characterizes the homological actions of framed mapping class groups. Refer to Section \ref{section:Xhom} for full details. Let $\phi$ be a framing of $(\Sigma_g, Z)$; measuring the winding number of a curve with respect to this framing gives rise to a ``winding number function'' from simple closed curves to $\Z$ (see \S\ref{subsection:wn}). The crossed homomorphism $\Theta_\phi$ of Theorems \ref{theorem:main}, \ref{theorem:corollary}, and \ref{theorem:main2} is then induced by measuring the ``change in winding number mod 2,'' a construction which generalizes the notion of a classical spin structure. 

We recall that a {\em classical spin structure} on a surface $\Sigma_g$ is a quadratic form $q: H_1(\Sigma_g;\Z/2\Z) \to \Z/2$, i.e., a function which satisfies $q(x+y) = q(x) + q(y)+ \pair{x,y}$, where $\pair{x,y}$ denotes the mod $2$ intersection pairing. Such $q$ can be used to determine a crossed homomorphism $\hat q: \Sp(2g,\Z/2\Z) \to H^1(\Sigma_g;\Z/2\Z)$ by the formula 
\[
\hat q (A)(x) = q(Ax) - q(x) \pmod 2,
\] 
which measures the change in $q$--value of each homology class.

Not all framings induce classical spin structures. In particular, we find that in \S\ref{section:more} that $\Theta_\phi$ behaves very differently depending on the combinatorics of $\phi$. For each marked point $p_i$ of $(\Sigma, Z)$, let $\Delta_i$ denote a small counterclockwise loop encircling $p_i$ and set $\kappa_i = -1 - \phi(\Delta_i)$ (here $\phi$ is viewed as a winding number function). Set $\sing = (\kappa_1, \ldots, \kappa_n)$.

If $\phi$ is a framing with all elements of $\sing$ even then the winding number function descends to a $\Z/2\Z$--valued winding number function on $H_1(\Sigma_g, \Z/2\Z)$; the change in winding number then induces a classical spin structure $q$. In this case, we show in Proposition \ref{proposition:summary} there is an equality
\[
\Theta_\phi = p^*(\hat q);
\]
here $p: \PAut(H_1(\Sigma_g, Z;\Z)) \to \Sp(2g,\Z/2\Z)$ is induced by the restriction to absolute homology followed by the reduction of coefficients mod 2.

If some element of $\sing$ is odd, then $\phi$ does not induce a classical spin structure and $\Theta_\phi$ is instead ``concentrated'' on the action on {\em relative} homology.
To describe this action, we note that we can write $\PAut(H_1(\Sigma_g, Z;\Z))$ as the extension of $\Sp(2g, \Z)$ by
\[
 \Hom(\widetilde H_0(Z;\Z), H_1(\Sigma_g;\Z)),
\]
which measures the transvection of the relative homology by absolute classes (see Section \ref{section:structure}). 

Define the element $v_{\sing} \in H_0(Z;\Z)$ by $v_{\sing} := \sum \kappa_i p_i$. This in turn defines a homomorphism
\[
v_{\sing}^*: \Hom(\widetilde H_0(Z;\Z), H_1(\Sigma_g;\Z)) \to H^1(\Sigma_g;\Z/2\Z)
\]
by the formula
\[
v_{\sing}^*(A)(x) = \pair{A(v_{\sing}),x} \pmod 2.
\]

When $\sing$ has odd elements, we show in Lemma \ref{lemma:ptpush} that $\Theta_\phi$ agrees with $v_{\sing}$ on  $\Hom(\widetilde H_0(Z;\Z), H_1(\Sigma_g;\Z))$, which in turn leads to the characterization of $\ker(\Theta_\phi)$ in terms of the short exact sequence
\[1 \to \ker(v_{\sing}^*) \to \ker(\Theta_\phi) \to \Sp(2g,\Z) \to 1.\]

\para{Outline of the proof} By replacing each boundary component of $\Sigma_{g,n}$ with a marked point, there is a map $p: \Mod(\Sigma_{g,n}) \to \PMod(\Sigma_g,Z)$ inducing a map $p_\phi: \Mod(\Sigma_{g,n})[\phi] \to \PMod(\Sigma_g,Z)[\phi]$; in the former we consider the {\em relative} framed mapping class group but in the latter we do not. The map $p_\phi$ is generally not surjective \cite[Proposition 6.11]{strata3}, but to prove both Theorems \ref{theorem:main} and \ref{theorem:main2} it will suffice to 
\begin{enumerate}
\item Construct  the crossed homomorphism $\Theta_\phi$ on $\PAut(H_1(\Sigma_g, Z; \Z))$ and show the containment $\Psi^{rel}(\PMod(\Sigma_g,Z)[\phi]) \le \ker(\Theta_\phi)$,
\item Show that $\Psi^{rel}(p_\phi(\Mod(\Sigma_{g,n}[\phi]))) = \ker(\Theta_\phi)$.
\end{enumerate}
Step (1) is carried out in Section \ref{section:Xhom}, where we define $\Theta_\phi$ as a measure of ``change of mod $2$ winding number'' for simple closed curves. The construction of $\Theta_\phi$ necessitates the discussions in Sections \ref{section:framings} and \ref{section:level2}, where we respectively discuss how a framing gives rise to a ``winding number function'' on simple closed curves, and some basic properties of the ``level $2$ mapping class group'' used to study the set of simple closed curves in a fixed mod $2$ homology class. 

Starting with a purely geometric definition of $\Theta_\phi$ as a function from $\PMod(\Sigma_g,Z)$ to a certain set, we show in Lemmas \ref{proposition:oddtrace} and \ref{proposition:induced} that $\Theta_\phi$ actually has the structure of a crossed homomorphism and that it is induced from a crossed homomorphism on $\PAut(H_1(\Sigma_g, Z; \Z))$. From the geometric origins of $\Theta_\phi$, it is then clear that $\PMod(\Sigma_g, Z)[\phi]$ is contained in the kernel. At the heart of these arguments are the ``twist--linearity'' and ``homological coherence'' properties of winding number functions discussed in Lemma \ref{lemma:HJ}. 

Step (2) is carried out in Section \ref{section:action}. The core result there is Proposition \ref{proposition:Taction}, which describes the action of the ``relative Torelli group'' $\mathcal I^{rel}(\Sigma_{g,n})$ on the set of relative framings. The proof of Theorems \ref{theorem:main} and \ref{theorem:main2} conclude with Proposition \ref{proposition:piece1}, which establishes the surjectivity of $\Psi^{rel}(\Mod(\Sigma_{g,n})[\phi])$ onto $\ker(\Theta_\phi)$. The strategy here is to first find any mapping class $f$ realizing an element $A \in \ker(\Theta_\phi)$, and use Proposition \ref{proposition:Taction} to adjust $f$ so as to stabilize $\phi$ without altering $\Psi^{rel}(f)$. 

Finally in Section \ref{section:more}, we give a more explicit description of the group $\ker(\Theta_\phi)$, emphasizing the difference in its structure caused by arithmetic properties of the framing $\phi$ (or equivalently, the arithmetic of the partition $\sing$ of $2g-2$). 

\subsection{Acknowledgments}

This project was begun when the authors were visiting MSRI for the Fall 2019 program ``Holomorphic Differentials in Mathematics and Physics.'' Both authors would like to thank the venue for its hospitality, excellent working environment, and generous travel support. The first author gratefully acknowledges support for this visit from NSF grants DMS-161087 as well as DMS-1107452, -1107263, and -1107367 ``RNMS: Geometric Structures and Representation Varieties'' (the GEAR Network).

We would also like to acknowledge Alex Wright for some very useful feedback on a preliminary draft, as well as an anonymous referee for a careful reading and helpful suggestions.

\section{Framings and framed mapping class groups}\label{section:framings}
 We briefly recall here the notion of a relative framing of a surface and the associated framed mapping class group. For a more thorough discussion, see \cite[Section 2]{strata3}. Throughout this section, we will formulate our results in the setting of surfaces with boundary. The theory of framings on a marked surface $(\Sigma_g, Z)$ exactly parallels the theory of ``absolute'' (i.e. non-relative) framings on $\Sigma_{g,n}$; we trust the reader can make the cosmetic adjustments necessary to formulate results for framings of $(\Sigma_g,Z)$. 

\subsection{(Relative) framings on surfaces with boundary} Let $\Sigma_{g,n}$ be a surface with $n \ge 1$ boundary components; then a {\em framing} of $\Sigma_{g,n}$ is a trivialization $\phi$ of the tangent bundle of $\Sigma_{g,n}$. After fixing a Riemannian metric $\mu$ on $\Sigma_g$ once and for all, a framing $\phi$ corresponds to a nowhere--vanishing vector field $\xi_\phi$.  We say that framings $\phi$ and $\psi$ are {\em isotopic} if $\xi_\phi$ and $\xi_\psi$ are homotopic through nowhere-vanishing vector fields. If $\phi$ and $\psi$ both restrict to the same framing $\delta$ of $\partial \Sigma_{g,n}$, then $\phi$ and $\psi$ are {\em relatively isotopic} if $\xi_\phi$ and $\xi_\psi$ are isotopic through vector fields restricting to $\delta$ on $\partial \Sigma_{g,n}$. 

\subsection{Winding number functions}\label{subsection:wn} The data of a (relative) isotopy class of framing is equivalent to a structure known as a {\em (relative) winding number function}, which is easier to work with in practice. We first observe that if $\gamma: S^1 \to \Sigma_{g,n}$ is a $C^1$ immersion, then the framing $\phi$ assigns a winding number $\phi(\gamma) \in \Z$ measuring the winding number of the forward--pointing tangent vector $\gamma'(t)$ with respect to $\xi_\phi$. It is not hard to see that $\phi(\gamma)$ is invariant under ambient isotopy. Let $\mathcal S(\Sigma_{g,n})$ denote the set of isotopy classes of oriented simple closed curves on $\Sigma_{g,n}$; then the framing $\phi$ determines a {\em winding number function}
\[
\phi: \mathcal S(\Sigma_{g,n}) \to \Z, \qquad c \mapsto \phi(c).
\]
Suppose now that each component $\Delta_i \subset \partial \Sigma_{g,n}$ is equipped with a point $p_i \in \Delta_i$ such that $\xi_\phi$ is orthogonally outward-pointing (such points $p_i$ always exist, possibly after adjusting $\xi_\phi$ by an isotopy supported near $\partial \Sigma_{g,n}$). We call such a point $p_i$ a {\em legal basepoint}. Choose exactly one legal basepoint on each boundary component. We say that a properly embedded arc $a: [0,1] \to \Sigma_{g,n}$ is {\em legal} if $a(0)$ and $a(1)$ are distinct legal basepoints, $a'(0)$ is orthogonally inward--pointing, and $a'(1)$ is orthogonally outward--pointing. The winding number of a legal arc is necessarily half--integral and is well--defined up to isotopy through legal arcs. Therefore, a framing $\phi$ gives rise to a {\em relative winding number function}
\[
\phi: \mathcal S^+(\Sigma_{g,n}) \to \tfrac{1}{2}\Z;\qquad c \mapsto \phi(c)
\]
where $\mathcal S^+(\Sigma_{g,n})$ denotes the set obtained from $\mathcal S(\Sigma_{g,n})$ by including all isotopy classes of legal arcs.

We say that the {\em signature} of a boundary component $\Delta \subset \partial \Sigma_{g,n}$ is the value $\phi(\Delta)$. On framed surfaces with marked points, the signature of a marked point is the winding number of a small counterclockwise loop encircling the marked point.

\subsection{(Relative) isotopy classes of framings}\label{section:DGB} The basic theory of relative isotopy classes of framings was established by Randal--Williams \cite{RW}. To state his results, we define a {\em distinguished geometric basis} on $\Sigma_{g,n}$ to be a collection
\[
\mathcal B = \{x_1,y_1, \dots, x_g, y_g\} \cup \{a_2, \dots, a_n\}
\]
of oriented simple closed curves $x_1, \dots, y_g$ and legal arcs $a_2, \dots, a_n$, subject to the following conditions. Below, the function $i(\cdot, \cdot)$ denotes the geometric intersection number, and $\pair{\cdot, \cdot}$ denotes the algebraic intersection number.
\begin{enumerate}
\item $i(x_i,y_i) = \pair{x_i,y_i} = 1$ and each $x_i, y_i$ is disjoint from all other $x_j,y_j, a_k$. 
\item Each arc $a_i$ is a legal arc running from a fixed legal basepoint $p_1 \in \Delta_1$ to the legal basepoint $p_i \in \Delta_i$, and the collection of $a_i$ are pairwise disjoint except at the common endpoint $p_1$.
\end{enumerate}

The following is a summary of the basic theory of relative winding number functions and relative isotopy classes of framings. For further discussion, see \cite[Section 2]{strata3} or \cite{RW}.

\begin{proposition}\label{proposition:WNprops}
Fix $g \ge 2$ and $n \ge 1$, and let $\delta$ be a framing of $\partial \Sigma_{g,n}$. Let $\phi, \psi$ be framings of $\Sigma_{g,n}$ restricting to $\delta$ on $\partial \Sigma_{g,n}$.
\begin{enumerate}
\item $\phi$ and $\psi$ are (relatively) isotopic if and only if the associated (relative) winding number functions are equal.
\item The (relative) winding number functions $\phi$ and $\psi$ are equal if and only if there are equalities $\phi(b) = \psi(b)$ for all $b \in \mathcal B$, where $\mathcal B$ is any distinguished geometric basis.
\end{enumerate}
\end{proposition}

\subsection{Mapping class group orbits} Randal--Williams classifies the set of orbits of relative framings under the action of $\Mod(\Sigma_{g,n})$. He finds that for $g \ge 2$ there are always exactly two orbits, classified by an element of $\Z/2\Z$ known as the {\em (generalized) Arf invariant}. The orbit structure in the case of absolute framings is somewhat different and was treated by Kawazumi \cite{kawazumi}, but we do not need to discuss this here. 

\begin{definition}[Arf invariant; c.f. Section 2.2 of \cite{strata3} and Section 2.4 of \cite{RW}]
Let $\Sigma_{g,n}$ be a surface with $g \ge 2$ and $n \ge 1$, and let $\phi$ be a relative framing of $\Sigma_{g,n}$; we denote the associated relative winding number function by the same symbol. Let $\mathcal B = \{x_1, \dots, y_g, a_2, \dots, a_n\}$ be a distinguished geometric basis. Define the element
\begin{equation}\label{equation:arf}
\Arf(\phi, \mathcal B) = \sum_{i = 1}^g (\phi(x_i) + 1)(\phi(y_i) + 1) + \sum_{i = 2}^n (\phi(a_i) + \tfrac{1}{2}) (\phi(\Delta_i) + 1) \pmod 2.
\end{equation}
\end{definition}
The Arf invariant classifies $\Mod(\Sigma_{g,n})$-orbits of relative framings in the following sense. 

\begin{proposition}[c.f. Proposition 2.8, Theorem 2.9 of \cite{RW}]\label{proposition:ModArfaction}
If $\mathcal B, \mathcal B'$ are two distinguished geometric bases for $\Sigma_{g,n}$, then $\Arf(\phi, \mathcal B) = \Arf(\phi, \mathcal B')$; consequently we write simply $\Arf(\phi)$. If $\phi$ and $\psi$ are framings of $\Sigma_{g,n}$ restricting to the same framing of $\partial \Sigma_{g,n}$, then there exists $f \in \Mod(\Sigma_{g,n})$ such that $f \cdot \phi = \psi$ if and only if $\Arf(\phi) = \Arf(\psi)$. 
\end{proposition}

\subsection{Properties of relative winding number functions} Following Proposition \ref{proposition:WNprops}, we know that isotopy classes of framings can be studied by means of their relative winding number functions. The results below establish some essential properties of relative winding number functions which were identified by Humphries--Johnson in \cite{HJ}.

\begin{lemma}\label{lemma:HJ}
Let $\phi$ be a relative winding number function. Then $\phi$ satisfies the following properties:
\begin{enumerate}
\item \label{item:reversibility} (Reversibility) Let $\bar c$ denote the curve/arc $c$ with the opposite orientation. Then $\phi(\bar c) = - \phi(c)$.
\item \label{TL} (Twist-linearity) Let $c \in \mathcal S(\Sigma_{g,n})$ and $a \in \mathcal S^+(\Sigma_{g,n})$ be given. Then 
\[
\phi(T_c^k(a)) = \phi(a) + k\pair{a,c}\phi(c).
\]
\item \label{HC} (Homological coherence) Let $S \subset \Sigma_{g,n}$ be a subsurface with boundary components $c_1, \dots, c_k$. Orient each $c_i$ so that $S$ lies to the left. Then
\[
\sum_{i = 1}^n \phi(c_i) = \chi(S).
\]
\end{enumerate}
\end{lemma}

\section{The level 2 mapping class group}\label{section:level2}
We collect here some basic facts about the ``level $2$ mapping class group'' that will be used in the following section. Let $\Psi_2$ be the homomorphism
\[
\Psi_2: \PMod(\Sigma_g, Z) \to \Sp(2g, \Z/2\Z)
\]
obtained by reducing the symplectic representation $\Psi$ mod $2$. We define
\[
\PMod(\Sigma_g, Z)[2] = \ker(\Psi_2).
\]
We emphasize that, as we have defined it, $\PMod(\Sigma_g, Z)[2]$ is the full preimage of the ``classical'' level $2$ subgroup $\PMod(\Sigma_g)[2] \le \PMod(\Sigma_g)$. In particular, no constraints are placed on the action of $\PMod(\Sigma_g,Z)[2]$ on {\em relative} homology classes in $H_1(\Sigma_{g}, Z; \Z/2\Z)$. We adopt this definition (as opposed to anything more restrictive) because we see below in Proposition \ref{proposition:level2trans} that $\PMod(\Sigma_g,Z)[2]$ acts transitively on simple closed curves in a fixed mod-$2$ homology class.

\begin{proposition}\label{proposition:level2trans}
Let $c, c' \subset \Sigma_g \setminus Z$ be simple closed curves, and suppose that $[c] = [c']$ as elements of $H_1(\Sigma_g; \Z/2)$. Then there exists $g \in \PMod(\Sigma_g,Z)[2]$ such that $g(c) = c'$.
\end{proposition}
\begin{proof}
Let $\bar c, \bar{c'}$ denote the images of $c, c'$ in $\Sigma_g$. It is a folklore result that for $\Sigma_g$ a closed surface, $\PMod(\Sigma_g)[2]$ acts transitively on the set of simple closed curves $c$ in a fixed mod $2$ homology class (compare \cite[Proposition 6.14]{FarbMarg}), and thus there exists $f \in \PMod(\Sigma_g)[2]$ such that $f(\bar c) = \bar{c'}$. Let $\tilde f$ be an arbitrary lift of $f$ to $\PMod(\Sigma_g, Z)$. By our definition of $\PMod(\Sigma_g,Z)[2]$, we have $\tilde f \in \PMod(\Sigma_g,Z)[2]$, and by construction, the curve $\tilde f(c)$ is isotopic to $c'$ after forgetting the set $Z$ of marked points. Let $p: \PMod(\Sigma_g,Z) \to \PMod(\Sigma_g)$ be the forgetful map; then there is an element $h \in \ker(p)$ such that $h(\tilde f(c))$ and $c'$ are isotopic rel $Z$. Therefore $g = h\tilde f \in \PMod(\Sigma_g,Z)[2]$ is the required element. 
\end{proof}

\begin{proposition}\label{proposition:level2gens}
For $g \ge 1$, the level $2$ mapping class group $\PMod(\Sigma_g,Z)[2]$ is generated by two classes of elements: ``squared twists'' $T_a^2$, and ``point--push maps'' $T_a T_b^{-1}$, where $a \cup b$ bounds an annulus containing a single element of $Z$.
\end{proposition}
\begin{proof}
According to \cite[Proposition 2.1]{humphries}, the closed level-$2$ mapping class group $\Mod(\Sigma_g)[2]$ is generated by the set of square-twists. By definition, $\PMod(\Sigma_g,Z)[2] = p^{-1}(\PMod(\Sigma_g)[2])$, where $p: \PMod(\Sigma_g,Z) \to \PMod(\Sigma_g)$ is the forgetful map. The kernel $\ker(p) = \PB(\Sigma_g,Z)$ is the {\em pure surface braid group on $n$ strands}. It is a classical fact (essentially a consequence of the Fadell-Neuwirth fibration; c.f. \cite[Section 9.1]{FarbMarg}) that $\PB(\Sigma_g,Z)$ is generated by point--push maps.
\end{proof}

\section{From framings to crossed homomorphisms}\label{section:Xhom}
In this section we begin the proof of Theorems \ref{theorem:main} and \ref{theorem:main2} in earnest. In Lemma \ref{proposition:oddtrace}, we use the winding number function associated to $\phi$ to define what turns out to be a crossed homomorphism on $\PMod(\Sigma_g,Z)$. In Lemma \ref{proposition:induced}, we show that this crossed homomorphism is pulled back from a crossed homomorphism $\Theta_\phi$ on $\PAut(H_1(\Sigma_g,Z;\Z))$. In the intermediate Section \ref{section:structure}, we present some basic results about the structure of $\PAut(H_1(\Sigma_g,Z;\Z))$ needed in the sequel. 

\subsection{A crossed homomorphism on the mapping class group}
\begin{lemma}\label{proposition:oddtrace}
Let $\phi$ be a framing of $(\Sigma_g, Z)$, and let $\Delta_\phi: \PMod(\Sigma_g, Z) \times \mathcal S \to \Z/2\Z$ be defined by
\[
\Delta_\phi(f,c) = \phi(f(c)) - \phi(c) \pmod 2.
\]
Then $\Delta_\phi$ determines a crossed homomorphism 
\[
\overline{\Delta_\phi}: \PMod(\Sigma_g,Z) \to H^1(\Sigma_g;\Z/2\Z)
\]
by the formula
\[
\overline{\Delta_\phi}(f)([c]) = \Delta_\phi(f,c).
\]
\end{lemma}

\begin{remark}\label{remark:kernel}
Observe that by construction, $\PMod(\Sigma_g,Z)[\phi] \le \ker(\overline{\Delta}_\phi)$.
\end{remark}

\begin{proof}
We begin with a simple but crucial observation: $\Delta_\phi$ satisfies a cocycle condition. For $f,g \in \PMod(\Sigma_g,Z)$ and any $c \in \mathcal S$, it follows easily from the definition that
\begin{equation}\label{equation:cocycle}
\Delta_\phi(fg,c) =  \Delta_\phi(f, g(c)) + \Delta_\phi(g, c).
\end{equation}

We divide the remainder of the proof into two steps. 

\para{Convention} Throughout this section, all arithmetic is taken mod $2$. The occasional presence of minus signs serves to help the reader navigate the logic of the calculations.

\para{Step 1: Descending to mod $\mathbf{2}$ homology} We suppose that $c, c' \in \mathcal S$ satisfy $[c] = [c']$ in $H_1(\Sigma_g;\Z/2\Z)$, and we wish to show that $\Delta_\phi(f,c) = \Delta_\phi(f, c')$ for $f \in \PMod(\Sigma_g,Z)$ arbitrary. If $[c] = [c']$, then by Proposition \ref{proposition:level2trans}, $c' = g(c)$ for some $g \in \PMod(\Sigma_g,Z)[2]$. Thus
\begin{align*}
\Delta_\phi(f, c') &= \Delta_\phi(f, g(c)) \\ &= \phi(fg(c)) - \phi(g(c))\\ &= \phi(fgf^{-1}(f(c))) - \phi(g(c)).
\end{align*}

By Proposition \ref{proposition:level2gens}, $g$ is a product of two kinds of mapping classes: squared twists $T_a^2$ and point--push maps. By the cocycle condition \eqref{equation:cocycle}, it suffices to examine the expression $\phi(fgf^{-1}(f(c))) - \phi(g(c))$ for $g$ one of these two forms. We claim that in either case,
\[
\phi(fgf^{-1}(f(c))) - \phi(g(c)) = \phi(f(c)) - \phi(c) = \Delta(f, c),
\]
thereby completing Step 1. Observe that for both classes of generators, $fgf^{-1}$ is an element of the same form as $g$. In the case $g = T_a^2$, the twist--linearity formula (Lemma \ref{lemma:HJ}.\ref{TL}) shows that
\begin{equation}\label{formula:squaretwist}
\phi(T_a^2(c)) = \phi(c) + 2 \pair{c,a}\phi(a) = \phi(c) 
\end{equation}
and therefore also $\phi(fgf^{-1}(f(c))) = \phi(f(c))$.

Suppose now that $g = T_{a_i} T_{a_i'}^{-1}$ with $a_i,a_i'$ cobounding an annulus containing the marked point $p_i$ of signature $-1-\kappa_i$. Then $[a_i] = [a_i']$ in $H_1(\Sigma_g;\Z/2\Z)$. The homological coherence property (Lemma \ref{lemma:HJ}.\ref{HC}) shows that $\phi(a_i) + \phi(a_i') = \kappa_i$. Applying the twist-linearity formula (Lemma \ref{lemma:HJ}.\ref{TL}), we find
\begin{equation*}
\phi(T_{a_i} T_{a_i'}^{-1}(c)) = \phi(c) + (\phi(a_i) + \phi(a_i'))\pair{[a_i], c} = \phi(c) + \kappa_i \pair{[a_i],c}.
\end{equation*}
and so
\begin{equation}\label{formula:ptpush}
\Delta_\phi(T_{a_i} T_{a_i'}^{-1},c) = \kappa_i \pair{[a_i],c}.
\end{equation}
Since $\pair{\cdot, \cdot}$ is invariant under the action $\Sp(2g,\Z/2\Z)$, this computation also shows that
\[
\Delta_\phi(fT_{a_i} T_{a_i'}^{-1}f^{-1},f(c)) = \kappa_i \pair{f([a_i]),f(c)} = \kappa_i \pair{[a_i],c} = \Delta_\phi(T_{a_i}T_{a_i'}^{-1},c)
\]
as required. 

\begin{remark}\label{remark:stronger}
We note here for later use that this argument actually establishes something stronger, namely, that if the signature of each marked point is odd then the value $\phi(c) \pmod 2$ is well-defined as a function on $H_1(\Sigma_g;\Z/2\Z)$. In particular, observe that \eqref{formula:squaretwist} shows that $\phi(T_a^2(c)) = \phi(c)$ for arbitrary curves $a, c$. When the signature of each marked point is odd then the $\kappa_i$ are even, and  \eqref{formula:ptpush} then shows that $\phi(T_a T_b^{-1}(c)) = \phi(c)$ as well.
\end{remark}

\para{Step 2: Additivity}  Following Step 1, we have established that $\Delta_\phi$ descends to a set-theoretic map 
\[
\overline{\Delta_\phi}: \PMod(\Sigma_g, Z) \times H_1(\Sigma_g; \Z/2\Z) \to \Z/2\Z
\]
that satisfies the cocycle condition \eqref{equation:cocycle}. In this step we complete the process of showing that $\Delta_\phi$ is a crossed homomorphism by showing that $\Delta_\phi$ is additive in the second argument. We fix $f \in \PMod(\Sigma_g,Z)$ and choose $x, y \in H_1(\Sigma_g; \Z/2\Z)$. There are two cases to consider: either $\pair{x,y} = 1$ or else $\pair{x,y} = 0$.

Suppose first that $\pair{x,y} = 1 \pmod 2$. Represent $x,y$ by simple closed curves $a,b$ satisfying $i(a,b) = 1$; then $[T_a(b)] = x+y$. By Step 1 and the cocycle condition \eqref{equation:cocycle},
\[
\Delta_\phi(f, x+y) = \Delta_\phi(f, T_a(b)) = \Delta_\phi(f T_a, b) + \Delta_\phi(T_a,b).
\]
We find that 
\begin{align*}
\Delta_\phi(f T_a, b) &=  \phi(f T_a(b)) - \phi(b) \\&= \phi(f T_a f^{-1} (f(b))) - \phi(b) \\&= \Delta_\phi(T_{f(a)}, f(b)) + \Delta_\phi(f,b).
\end{align*}
To evaluate the expressions $\Delta_\phi(T_{f(a)}, f(b))$ and $\Delta_\phi(T_a,b)$, we appeal to the definition of $\Delta_\phi$ and the twist--linearity formula (Lemma \ref{lemma:HJ}.\ref{TL}):
\[
\Delta_\phi(T_a,b) = \phi(T_a(b)) - \phi(b) = \pair{b,a}\phi(a) =\phi(a).
\]
Likewise,
\[
\Delta_\phi(T_{f(a)}, f(b))  = \pair{f(b),f(a)}\phi(f(a)) = \phi(f(a)),
\]
since $\pair{f(b), f(a)} = \pair{b,a} = 1$. Altogether, we have shown
\begin{align*}
\Delta_\phi(f, T_a(b)) &= \Delta_\phi(f T_a, b) + \Delta_\phi(T_a,b) \\&= \phi(f(a)) - \phi(a) + \Delta_\phi(f,b) \\&= \Delta_\phi(f,a) + \Delta_\phi(f,b)
\end{align*}
as required.

The other case $\pair{x,y} = 0$ proceeds similarly, replacing $T_a$ with a different mapping class. Given $x,y$ satisfying $\pair{x,y} = 0$, represent $x,y$ by simple closed curves $a,b$ satisfying $i(a,b) = 0$. Let $c$ be a curve such that $a \cup b \cup c$ bounds a pair of pants on $(\Sigma_g,Z)$ containing no marked points. Then $[c] = x+y$. We choose $g \in \PMod(\Sigma_g,Z)$ such that $g(a) = c$. As before,
\[
\Delta_\phi(f, x+ y) = \Delta_\phi(f,g(a)) = \Delta_\phi(fg, a) + \Delta_\phi(g,a).
\]
We analyze the first summand $\Delta_\phi(fg,a)$ as before:
\[
\Delta_\phi(fg,a) = \phi(fg(a)) - \phi(a) = \phi(fgf^{-1}(f(a))) - \phi(a) = \Delta_\phi(fgf^{-1},f(a)) + \Delta_\phi(f,a). 
\]
We are left with evaluating the expressions $\Delta_\phi(g,a)$ and $\Delta_\phi(fgf^{-1},f(a))$. The following expression holds by homological coherence (Lemma \ref{lemma:HJ}.\ref{HC}):
\[
\phi(a) + \phi(b) + \phi(c) = 1.  
\]
Thus,
\begin{equation}\label{ga}
\Delta_\phi(g,a) = \phi(g(a)) - \phi(a) = \phi(c) - \phi(a) = \phi(b) + 1.
\end{equation}

The same relation holds among the $\phi$-values of $f(a), f(b), f(c)$, showing that
\begin{equation}\label{fgfa}
\Delta_\phi(fgf^{-1}, f(a)) = \phi(f(b)) + 1.
\end{equation}
Adding together the three contributions $\Delta_\phi(f,a)$, \eqref{ga}, \eqref{fgfa} to $\Delta(f,g(a))$, we have
\begin{align*}
\Delta_\phi(f, x+y) &= \Delta_\phi(f, ga) \\&= \Delta_\phi(f,a) + \phi(f(b)) + 1 + \phi(b) + 1 \\&= \Delta_\phi(f,a) + \Delta_\phi(f,b)
\end{align*}
as required.
\end{proof}

\subsection{The structure of PAut$\mathbf{(H_1(\Sigma_g,Z;\Z))}$} \label{section:structure} Before proceeding to the second key result of the section (Lemma \ref{proposition:induced}), we must first consider some of the basic structural properties of the groups $\PMod(\Sigma_g,Z)$ and $\PAut(H_1(\Sigma_g,Z;\Z))$. We recall that the long exact sequence for the pair $(\Sigma_g, Z)$ specializes to the following short exact sequence:
\[
0 \to H_1(\Sigma_g,\Z) \to H_1(\Sigma_g, Z;\Z) \to \widetilde H_0(Z;\Z) \to 0.
\]
The action of $\PMod(\Sigma_g,Z)$ on $H_1(\Sigma_g,Z;\Z)$ preserves the subspace $H_1(\Sigma_g,\Z)$ and the action there preserves the algebraic intersection pairing $\pair{\cdot, \cdot}$. Define
\[
\RelAut(H_1(\Sigma_g,Z;\Z)) := \Hom(\widetilde H_0(Z;\Z), H_1(\Sigma_g;\Z)).
\]
The group $\PAut(H_1(\Sigma_g,Z;\Z))$ is then characterized by the following short exact sequence:
\begin{equation}\label{equation:aut}
1 \to \RelAut(H_1(\Sigma_g,Z;\Z)) \to \PAut(H_1(\Sigma_g,Z;\Z)) \to \Sp(2g,\Z) \to 1.
\end{equation}
Note that with the definitions above, we only consider automorphisms inducing a trivial action on $\widetilde H_0(Z;\Z)$. If we were to consider $\Psi^{rel}(\Mod(\Sigma_g,Z))$ (i.e. removing the purity assumption), we would need to enlarge $\RelAut(H_1(\Sigma_g,Z;\Z))$ by taking the semi-direct product with (a subgroup of) the symmetric group of $Z$.

\subsection{Crossed homomorphisms on PAut$\mathbf{(H_1(\Sigma_g,Z;\Z))}$}
\begin{lemma}\label{proposition:induced}
The crossed homomorphism $\overline{\Delta_\phi}$ is induced from a crossed homomorphism
\[
\Theta_\phi: \PAut(H_1(\Sigma_g, Z;\Z)) \to H^1(\Sigma_g; \Z/2\Z).
\]
\end{lemma}

\begin{proof}
Define the ``relative Torelli group''
\[
\mathcal I^{rel}(\Sigma_g,Z) := \ker \Psi^{rel}.
\]
The claim amounts to showing that the restriction of $\overline{\Delta_\phi}$ to $\mathcal I^{rel}(\Sigma_g, Z)$ is trivial. We consider the following commutative diagram of short exact sequences, where the rows are induced by the forgetful map $(\Sigma_g, Z) \to \Sigma_g$ and the columns are induced by considering the action of the (relative/absolute) mapping class group on (relative/absolute) homology. By definition, $K := \ker( \Psi^{rel} \mid_{\PB(\Sigma_g, Z)})$.
\[
\xymatrix{
		&	1 \ar[d]								&	1 \ar[d]							&	1 \ar[d]					&	\\
1 \ar[r]	& K \ar[r] \ar[d]								& \mathcal I^{rel}(\Sigma_g,Z) \ar[r]\ar[d]		& \mathcal I(\Sigma_g) \ar[r]\ar[d]	& 1	\\
1 \ar[r]	& \PB(\Sigma_g, Z) \ar[r] \ar[d]^{\Psi^{rel}}			& \PMod(\Sigma_g,Z) \ar[r] \ar[d]^{\Psi^{rel}}	& \PMod(\Sigma_g) \ar[r] \ar[d]^{\Psi}	&1	\\
1 \ar[r]	& \RelAut(H_1(\Sigma_g,Z;\Z))\ar[r] \ar[d]			& \PAut(H_1(\Sigma_g,Z;\Z)) \ar[r] \ar[d]		& \Sp(2g,\Z)	\ar[r] \ar[d]	& 1	\\
		&	1									& 	1								&	1						&	
}
\]

To prove the claim, it suffices to show that $\overline{\Delta_\phi}(f) = 0$ for $f \in \mathcal I^{rel}(\Sigma_g,Z)$ a generator. Below, we determine a generating set for $\mathcal I^{rel}(\Sigma_g,Z)$.

For $g \ge 3$, the Torelli group $\mathcal I(\Sigma_g)$ is generated by ``bounding pair maps:'' these are elements of the form $T_a T_b^{-1}$ where $a, b$ are disjoint simple closed curves on $\Sigma_g$ that bound a subsurface $S \subset \Sigma_g$. For any such pair, there are simple closed curves $\tilde a, \tilde b$ on $(\Sigma_g, Z)$ which are isotopic to $a$ and $b$ after forgetting $Z$ and so that $\tilde a \cup \tilde b$ bounds a subsurface $\tilde S \subset (\Sigma_g, Z)$ which does not contain points of $Z$. We call the corresponding mapping class $T_{\tilde a} T_{\tilde b}^{-1}$ a {\em strict bounding pair map}. It is easy to see that 
\[
T_{\tilde a} T_{\tilde b}^{-1} \in \mathcal I^{rel}(\Sigma_g,Z),
\]
and hence $\mathcal I^{rel}(\Sigma_g,Z)$ is generated by strict bounding pair maps and $K$. 

Let $T_{\tilde a} T_{\tilde b}^{-1}$ be a strict bounding pair map and let $c \subset (\Sigma_g, Z)$ be a simple closed curve. Since $[\tilde a] = [\tilde b]$ as elements of $H_1(\Sigma_g;\Z/2\Z)$, we find
\begin{eqnarray*}
\Delta_\phi(T_{\tilde a} T_{\tilde b}^{-1}, c) &=&  \phi(T_{\tilde a} T_{\tilde b}^{-1}(c)) - \phi(c)\\
								&=& \phi(c) + \pair{c, \tilde a} \phi(\tilde a) - \pair{c, \tilde b}\phi(b) - \phi(c)\\
								&=& \pair{c, \tilde a}(\phi(\tilde a) - \phi(\tilde b)).
\end{eqnarray*}
Since $\tilde a \cup \tilde b$ is a strict bounding pair, the bounded subsurface $\tilde S$ has two boundary components and no punctures and hence has even Euler characteristic. By homological coherence (Lemma \ref{lemma:HJ}.\ref{HC}), $\phi(\tilde a) = \phi(\tilde b) \pmod 2$. The above then shows that $\Delta_\phi(T_{\tilde a} T_{\tilde b}^{-1}) = 0$ as required.

It remains to show that $\Delta_\phi(f, c) = 0$ for $f \in K$ arbitrary. For this, it suffices to show that the restriction of $\Delta_\phi$ to $\PB(\Sigma_g,Z)$ factors through $\Psi^{rel}$. As noted in the proof of Proposition \ref{proposition:level2gens}, the group $\PB(\Sigma_g,Z)$ is generated by elements of the form $T_{a_i} T_{a_i'}^{-1}$, where $a_i \cup a_i'$ cobound an annulus on $\Sigma_g$ containing a unique point $p_i \in Z$. Formula \eqref{formula:ptpush} above exactly shows that $\Delta_\phi$ factors through $\Psi^{rel}$ for elements of this form, and the claim for $g \ge 3$ follows.

For $g = 2$, the Torelli group is instead generated by ``separating twist maps''. A similar argument shows that the result holds in this case as well. 
\end{proof}

\section{The action of Torelli on framings}\label{section:action}

\subsection{The Torelli action} The results of the previous section imply that if $f \in \PMod(\Sigma_g,Z)$ preserves a framing $\phi$, then $\Theta_\phi(\Psi^{rel}(f)) = 0$. To establish the opposite containment, we must understand how $\ker(\Psi^{rel})$ acts on the set of framings. With Theorem \ref{theorem:main2} in mind, we work here in the setting of a surface $\Sigma_{g,n}$ with $n \ge 1$ boundary components, equipped with a {\em relative} framing $\phi$.

If framings $\phi$ and $\psi$ are relatively isotopic, then by definition they restrict to the same framing of $\partial \Sigma_{g,n}$. Proposition \ref{proposition:Taction} below shows that in this case, orbits are classified by the data of a ``$q$-vector.'' To define this object, let $\mathcal B = \{x_1, \dots, y_g, a_2, \dots, a_n\}$ be a distinguished geometric basis for $\Sigma_{g,n}$ (c.f. Section \ref{section:DGB}). We call the simple closed curves $\{x_1, \dots, y_g\}$ the {\em absolute elements}, and the arcs $\{a_2, \dots, a_n\}$ the {\em relative elements}. The {\em $q$-vector} $\vec{q}(\mathcal B, \phi) \in (\Z/2\Z)^{2g}$ is the element
\[
\vec q(\mathcal B, \phi) = (\phi(x_1), \dots, \phi(y_g)) \pmod 2.
\]

\begin{proposition}\label{proposition:Taction}
Fix $g \ge 2$ and $n \ge 1$, and let $\phi$ and $\psi$ be relative framings of $\Sigma_{g,n}$ that restrict to the same framing of $\partial \Sigma_{g,n}$. Suppose moreover that $\Arf(\phi) = \Arf(\psi)$ and that there exists a distinguished geometric basis $\mathcal B$ such that $\vec{q}(\mathcal B, \phi) = \vec q(\mathcal B, \psi)$. Then there exists $f \in \mathcal I^{rel}(\Sigma_{g,n})$ such that $f \cdot \phi = \psi$. 
\end{proposition}

\begin{proof}
The strategy is as follows: starting with the distinguished geometric basis $\mathcal B$, we successively modify the $\psi$-winding numbers of elements $b \in \mathcal B$ while preserving the relative homology classes. Under the hypotheses of the proposition, we will find that it is possible to construct a new distinguished geometric basis $\mathcal B'$ such that $\phi(b) = \psi(b')$ for all pairs of corresponding elements $b \in \mathcal B, b' \in \mathcal B'$, and moreover $[b] = [b']$ in $H_1(\Sigma_{g,n}, \partial \Sigma_{g,n}; \Z)$ for all such pairs. Then a mapping class $f$ taking $\mathcal B$ to $\mathcal B'$ will necessarily act trivially on $H_1(\Sigma_{g,n}, \partial \Sigma_{g,n}; \Z)$ and will also satisfy $f \cdot \phi = \psi$. Throughout, we will discuss the notion of a ``marking'' of $H_1(\Sigma_{g,n}, \partial \Sigma_{g,n};\Z)$, by which we mean a specified ordered basis. 

	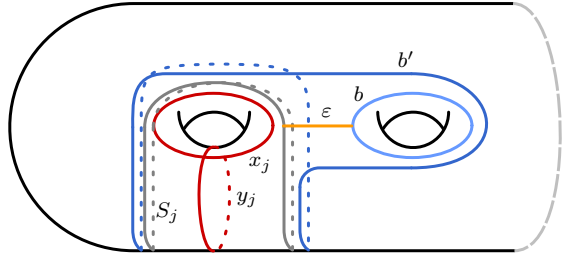
\begin{figure}[ht] 
		\labellist
		\Small
		\pinlabel $x_j$ at 95 33
		\pinlabel $y_j$ at 90 20
		\pinlabel $S_j$ at 60 15
		\pinlabel $\epsilon$ at 120 53
		\pinlabel $b$ at 132 60
		\pinlabel $b'$ at 150 73
		\endlabellist
		\includegraphics[scale=1]{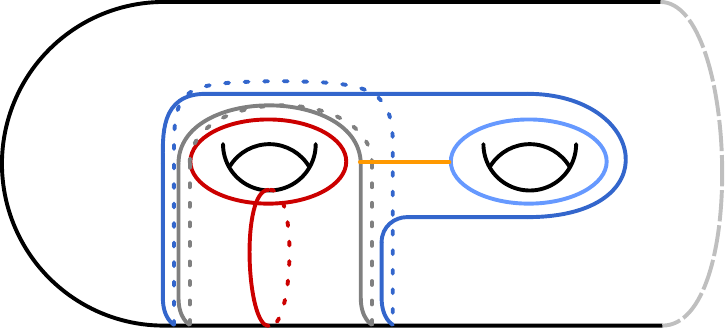}
		\caption{The connect-sum operation used in Step 1 of Proposition \ref{proposition:Taction}.}
		\label{figure:augmenteddual}
	\end{figure}

\para{Step 1: absolute elements} We first modify the absolute elements $\{x_1, \dots, y_g\} \in \mathcal B$ using an operation called a ``connect--sum,'' defined as follows: given some absolute element $b$ of the form $b = x_i, y_i$, we choose an index $j \ne i$ and consider the boundary $S_j$ of a regular neighborhood of $x_j \cup y_j$. It is possible to choose an arc $\epsilon$ connecting the left side of $b$ to $S_j$ that is otherwise disjoint from all elements of $\mathcal B$. The connect--sum $b'$ is then the curve obtained as the third boundary component of a regular neighborhood of $b \cup \epsilon \cup S_j$ (the other boundary components being $b$ and $S_j$).  By the homological coherence property (Lemma \ref{lemma:HJ}.\ref{HC}), $\phi(S_i) = 1$ when oriented with $x_j, y_j$ to the right (i.e. with $b'$ to the left). Applying homological coherence to the pair of pants $P$ bounded by $b, b', S_j$, we find
\[
\phi(b') = \phi(b) + \phi(S_j) + 1 = \phi(b) +2
\]
when $b'$ is oriented with $P$ to the right (i.e. when $[b] = [b']$). Likewise, if one takes $\epsilon$ to be an arc connecting the {\em right} side of $b$ to $S_j$, the resulting $b'$ satisfies $\phi(b') = \phi(b) - 2$.

The set of curves and arcs obtained by replacing $b$ with $b'$ is still a distinguished geometric basis determining the same marking of $H_1(\Sigma_{g,n}, \partial \Sigma_{g,n};\Z)$. By hypothesis, $\vec q(\mathcal B, \phi) = \vec q(\mathcal B, \psi)$ and so $\phi(b) = \psi(b) \pmod 2$. Thus by repeated use of the above operations, we can replace $\mathcal B$ with a new distinguished geometric basis $\mathcal B_{abs}'$ such that $\phi(b) = \psi(b')$ for each pair of matching absolute elements $b, b'$. 

\para{Step 2: relative elements} It remains to adjust the relative elements $a_2, \dots, a_n$ of $\mathcal B_{abs}'$ so that their $\psi$--winding numbers match the $\phi$--winding numbers of the relative elements of $\mathcal B$. By hypothesis, $\Arf(\phi) = \Arf(\psi)$. Since also $\vec{q}(\mathcal B, \phi) = \vec q(\mathcal B, \psi)$, it follows that the ``relative Arf invariants''
\[
\RelArf(\gamma) = \sum_{i =2}^n (\gamma(a_i) + \tfrac{1}{2})(\gamma(\Delta_i) + 1) \pmod 2
\]
must be equal for $\gamma = \phi, \gamma = \psi$. The equality of relative Arf invariants implies that there must be an {\em even} number of indices $j$ for which $\phi(\Delta_i) = \psi(\Delta_i)$ is even and $\phi(a_j) \ne \psi(a_j)$.

For a pair of indices $j_1, j_2$ with $\psi(\Delta_{j_1})$ and $\psi(\Delta_{j_2})$ both even, we choose a curve $d$ such that $i(d,b) = 1$ if $b = a_{j_1}$ or $b = a_{j_2}$, and otherwise equals zero. Such $d$ necessarily bounds a pair of pants with the other boundary curves given by $\Delta_{j_1}$ and $\Delta_{j_2}$. By hypothesis, since $\psi(\Delta_{j_1})$ and $\psi(\Delta_{j_2})$ are both even, the homological coherence property (Lemma \ref{lemma:HJ}.\ref{HC}) implies that $\psi(d)$ is odd. Thus applying the twist $T_d$ alters the parities of $\psi(a_{j_1}), \psi(a_{j_2})$ while leaving all other $\psi$-values, as well as the marking of $H_1(\Sigma_{g,n}, \partial \Sigma_{g,n};\Z)$, fixed. 

If $\psi(\Delta_j)$ is odd, then the parity of $\psi(a_j)$ can be altered by applying $T_{\Delta_j}$; again this leaves the other $\psi$ values and the marking fixed. At this point, there are equalities $\phi(a_i) = \psi(a_i) \pmod 2$ for all relative elements $a_i$. By performing a series of connect--sum moves to the arcs as in the absolute case, we can adjust the integral values of $\psi(a_i)$ in increments of $2$ while preserving the homology classes. At the end of this procedure, we have produced a distinguished geometric basis $\mathcal B'$ such that $\phi(b) = \psi(b')$ for all pairs of matching elements $b \in \mathcal B, b' \in \mathcal B'$, and such that $\mathcal B$ and $\mathcal B'$ determine the same marking. Letting $f \in \Mod(\Sigma_{g,n})$ be such that $f(\mathcal B) = \mathcal B'$, we find that $f \in \mathcal I^{rel}(\Sigma_{g,n})$ and $f \cdot \phi = \psi$ as required.
\end{proof}

\subsection{Concluding the proof of Theorems \ref{theorem:main} and \ref{theorem:main2}} We are now in a position to complete the proof of Theorems \ref{theorem:main} and \ref{theorem:main2}. As discussed in the introduction, there is a homomorphism $p_\phi: \Mod(\Sigma_{g,n})[\phi] \to \PMod(\Sigma_g,Z)[\phi]$, and so it suffices to establish the result in the setting of $\Mod(\Sigma_{g,n})[\phi]$. This follows from Proposition \ref{proposition:piece1} below.

\begin{proposition}\label{proposition:piece1}
Let $\phi$ be a relative framing of $\Sigma_{g,n}$. Then there is an equality
\[
\Psi^{rel}(\Mod(\Sigma_{g,n}[\phi])) = \ker (\Theta_\phi).
\]
\end{proposition}
\begin{proof}
The containment $\Psi^{rel}(\Mod(\Sigma_{g,n})[\phi]) \le \ker(\Theta_\phi)$ is clear from the properties of $\Theta_\phi$ established in Section \ref{section:Xhom} (c.f. Remark \ref{remark:kernel} and Lemma \ref{proposition:induced}).

To establish the opposite containment, let $A \in \ker(\Theta_\phi)$ be given; then $A$ lifts to a mapping class $f \in \Mod(\Sigma_{g,n})$. By Lemmas \ref{proposition:induced} and \ref{proposition:oddtrace}, $\phi(f(c)) = \phi(c) \pmod 2$ for every simple closed curve $c \subset \Sigma_g$. In particular, if $\mathcal B$ is any distinguished geometric basis, then $\vec q(\mathcal B, \phi) = \vec q(\mathcal B, f \cdot \phi)$.

As $\Mod(\Sigma_{g,n})$ preserves the Arf invariant (Proposition \ref{proposition:ModArfaction}), we have that $\Arf(\phi) = \Arf(f \cdot \phi)$. Now since $\phi$ and $f \cdot \phi$ restrict to the same framing of $\partial \Sigma_{g,n}$, we can apply Proposition \ref{proposition:Taction} to see that there is an element $g \in \mathcal I^{rel}(\Sigma_{g,n})$ such that $g \cdot (f \cdot \phi) = \phi$. The mapping class $gf$ satisfies $\Psi^{rel}(gf) = \Psi^{rel}(f) = A$ and $gf \cdot \phi = \phi$ as required.
\end{proof}

\section{More on $\Theta_\phi$ and its kernel}\label{section:more}
Theorems \ref{theorem:main} and \ref{theorem:main2} give a uniform description of the homological action of framed mapping class groups valid for arbitrary framings. We find it interesting that despite this, the crossed homomorphism $\Theta_\phi$ itself behaves quite differently depending on some arithmetic properties of $\phi$. For $i = 1, \dots, n$, define
\[
\kappa_i = -1 - \phi(\Delta_i);
\]
note that in the case where $\phi$ is induced from a differential in the stratum $\Omega \mathcal M_g(\sing)$, this agrees with the definition of $\kappa_i$ given there. Define $r = \gcd(\kappa_i)$.

To conclude this note, we offer below in Proposition \ref{proposition:summary} a more explicit description of $\ker(\Theta_\phi)$ which makes apparent the different structure that appears in the regimes of $r$ even and $r$ odd. Recall from Section \ref{section:structure} that by definition,
\[
\RelAut(H_1(\Sigma_g,Z;\Z)) = \Hom(\widetilde H_0(Z;\Z), H_1(\Sigma_g; \Z)).
\]
We define the element $v_{\sing} \in H_0(Z;\Z)$ by
$
v_{\sing} := \sum_{i=1}^n \kappa_i p_i.
$
We can use $v_{\sing}$ to define a homomorphism 
\[
v_{\sing}^*: \RelAut(H_1(\Sigma_g,Z;\Z)) \to H^1(\Sigma_g;\Z/2\Z)
\]
by the formula
\[
v_{\sing}^*(A)(x) = \pair{A(v_{\sing}),x} \pmod 2.
\]
Note that $v_{\sing} \pmod 2$, and hence $v_{\sing}^*$, is trivial if all $\kappa_i \in \sing$ are even.

\begin{lemma}\label{lemma:ptpush}
Restricted to $\RelAut(H_1(\Sigma_g,Z;\Z))$, there is an equality
\[
\Theta_\phi = v_{\sing}^*.
\]
In particular, if $r$ is even, then $\Theta_\phi$ is identically zero on $\RelAut(H_1(\Sigma_g, Z;\Z))$. 
\end{lemma}

\begin{proof} We evaluate the expression $\Theta_\phi(A)(x)$ for $A \in \RelAut(H_1(\Sigma_g,Z;\Z))$ and $x \in H_1(\Sigma_g;\Z/2\Z)$ arbitrary. By Proposition \ref{proposition:induced}, we are free to choose $f \in \PMod(\Sigma_g,Z)$ representing our given element $A \in \RelAut(H_1(\Sigma_g,Z;\Z))$, and by the cocycle condition \eqref{equation:cocycle}, it suffices to restrict attention to $A$ a member of some generating set for $\RelAut(H_1(\Sigma_g,Z;\Z))$. 

As
\[
\RelAut(H_1(\Sigma_g,Z;\Z)) \cong \Hom(\widetilde H_0(Z;\Z), H_1(\Sigma_g; \Z)),
\]
we see that $\RelAut(H_1(\Sigma_g,Z;\Z))$ is generated by elements $P_{i,a}$ sending the $i^{th}$ factor to a primitive element $a \in H_1(\Sigma_g;\Z)$ and acting trivially on the remaining factors. Such elements are represented by mapping classes $\Pi_{i,\tilde a}$ given as point-pushes $T_{\tilde a_i}T_{\tilde a_i'}^{-1}$ as in \eqref{formula:ptpush}. Thus we must evaluate $\Delta_\phi(\Pi_{i,a},\tilde c)$ for $\tilde c$ an arbitrary simple closed curve with $[\tilde c] = c$. Formula \eqref{formula:ptpush} shows that
\[
\Delta_\phi(\Pi_{i,\tilde a},\tilde c) = v_{\sing}^*(P_{i,a}(c)),
\]
and the result holds in general by linearity.
\end{proof}

To better understand the case of $r$ even, we observe that homological coherence (Lemma \ref{lemma:HJ}.\ref{HC}) implies that if $\phi$ is an framing of $(\Sigma_g, Z)$ with $r$ even, then the assignment 
\begin{equation}\label{classicalspin}
q(x) = \phi(\tilde x) + 1 \pmod 2,
\end{equation}
where $\tilde x$ is a simple closed curve with $[\tilde x] = x$ in $H_1(\Sigma_g;\Z/2\Z)$, determines a classical spin structure. See also Remark \ref{remark:stronger} or \cite[Theorem 1A]{johnson_spin} for more details.

We let $\Sp(2g,\Z)[q] \le \Sp(2g,\Z)$ denote the stabilizer of a classical spin structure $q$. This can be extended to define the subgroup 
\[
\PAut(H_1(\Sigma_g, Z;\Z))[q] \le \PAut(H_1(\Sigma_g, Z;\Z))
\]
as the full preimage of $\Sp(2g,\Z)[q]$ in $\PAut(H_1(\Sigma_g, Z;\Z))$.

\begin{lemma} \label{lemma:reven}
Suppose that $r$ is even, and define $q:= \phi \pmod 2+1$ as in \eqref{classicalspin}. Then there is an equality
\[
\ker(\Theta_\phi) = \PAut(H_1(\Sigma_g, Z;\Z))[q].
\]
\end{lemma}

\begin{proof}
We recall Remark \ref{remark:stronger}: if $r$ is even, then $q := \phi \pmod 2$ is well-defined as a function on $H_1(\Sigma_g;\Z/2\Z)$. Thus $\ker(\Theta_\phi) \le \PAut(H_1(\Sigma_g, Z; \Z))[q]$. To establish the opposite containment, it suffices to show that $\ker(\Theta_\phi)$ contains $\RelAut(H_1(\Sigma_g,Z;\Z))$, and this follows from the second assertion of Lemma \ref{lemma:ptpush}.
\end{proof}

Lemma \ref{lemma:reven} implies that the constraint imposed by $\Theta_\phi$ is ``concentrated'' on the action on absolute homology. For $r$ odd, we find that the opposite is true: Lemma \ref{lemma:ptpush} implies that $\Theta_\phi$ is nontrivial on $\RelAut(H_1(\Sigma_g,Z;\Z))$, while Lemma \ref{lemma:rodd} below shows that no constraints are imposed on the image of $\ker(\Theta_\phi)$ on $\Sp(2g,\Z)$.

\begin{lemma}\label{lemma:rodd}
For $r$ odd, there is a surjection $\ker(\Theta_\phi) \onto \Sp(2g,\Z)$. 
\end{lemma}
\begin{proof}
It suffices to show that $\PMod(\Sigma_{g}, Z)[\phi]$ surjects onto $\Sp(2g,\Z)$ via $\Psi^{rel}$. The group $\Sp(2g,\Z)$ is generated by {\em transvections}, automorphisms of the form $T_v(x) = x + \pair{x,v}v$, as $v$ ranges over the set of primitive elements in $H_1(\Sigma_g,\Z)$. Thus it suffices to construct, for such an arbitrary primitive element $v$, a Dehn twist $T_c$ such that $[c] = v$ in $H_1(\Sigma_g;\Z)$ and such that $T_c \in \PMod(\Sigma_{g},Z)[\phi]$. By the twist--linearity formula (Lemma \ref{lemma:HJ}.\ref{TL}), it is sufficient to construct a curve $c$ such that $[c] = v$ and such that $\phi(c) = 0$. 

Let $c'$ be an arbitrary simple closed curve representing $v$. If $\phi(c')$ is even, the techniques of Proposition \ref{proposition:Taction} can be used to construct a curve $c$ with $[c] = [c']= v$ and with $\phi(c) = 0$. Suppose then that $\phi(c')$ is odd. Since $r$ is odd, there is some point $p_i$ for which the associated $\kappa_i$ is odd. Let $\alpha$ be a simple path based at $p_i$ that crosses $c'$ exactly once, and let $\Pi_\alpha$ denote the associated point-pushing map. Then \eqref{formula:ptpush} implies that $\phi(\Pi_\alpha(c'))$ is even (and also $[\Pi_\alpha(c')] = [c']$), and hence the argument above can be applied to $\Pi_\alpha(c)$. 
\end{proof}

We summarize the results of this section as follows.

\begin{proposition}\label{proposition:summary}
For $r$ even, there is an isomorphism
\[
\ker(\Theta_\phi) \cong \Sp(2g,\Z)[q] \ltimes \RelAut(H_1(\Sigma_g, Z;\Z)),
\]
with $q$ the classical spin structure associated to $\phi$. For $r$ odd, $\ker(\Theta_\phi)$ can be described by the (non-split) extension
\begin{equation*}\label{rodd}
1 \to \ker(v_{\sing}^*) \to \ker(\Theta_\phi) \to \Sp(2g,\Z) \to 1.
\end{equation*}
\end{proposition}

\bibliographystyle{alpha}
\bibliography{library}

\end{document}